\documentclass[11pt]{article}
\date{}

\usepackage[title]{appendix}
\usepackage{xcolor}
\usepackage[margin=1in]{geometry}                
\usepackage{graphicx}
\usepackage{subcaption}
\usepackage{pdflscape}
\usepackage{amssymb}
\usepackage[normalem]{ulem}
\usepackage{hyperref}
\usepackage{enumitem}
\usepackage{mathrsfs}
\usepackage{epstopdf}
\usepackage{rotating}
\usepackage{longtable} 
\usepackage{adjustbox}
\usepackage{float}

\usepackage{color}
\usepackage{bbm, dsfont}
\usepackage{pst-node}
\usepackage{tikz-cd}
\usepackage{amsfonts} 
\usepackage{geometry}
\usepackage{amsthm}
\usepackage{amsmath}
\usepackage{titlesec}
\usepackage{amssymb}
\usepackage{enumitem}
\usepackage{float}
\usepackage [english]{babel}
\usepackage [autostyle, english = american]{csquotes}
\usepackage{algorithm}
\usepackage[noend]{algpseudocode} 
\makeatletter
\def\BState{\State\hskip-\ALG@thistlm}
\makeatother
\usepackage{hyperref}
\usepackage[normalem]{ulem}

\newlist{casess}{enumerate}{1}
\setlist[casess]{label=     \textbf{Case} \arabic*:}
\usepackage{mathtools}

\makeatletter
\newcommand*{\rom}[1]{\expandafter\@slowromancap\romannumeral #1@}
\makeatother

\usepackage{etoolbox}

\makeatletter
\patchcmd{\ttlh@hang}{\parindent\z@}{\parindent\z@\leavevmode}{}{}
\patchcmd{\ttlh@hang}{\noindent}{}{}{}
\makeatother

\usepackage{listings}
\usepackage{color} 
\definecolor{mygreen}{RGB}{28,172,0} 
\definecolor{mylilas}{RGB}{170,55,241}

\newlist{Assumptions}{enumerate}{1}
\setlist[Assumptions]{label=     \textbf{Assumption} \arabic*:}

\makeatletter

\newsavebox{\@brx}
\newcommand{\llangle}[1][]{\savebox{\@brx}{\(\m@th{#1\langle}\)}%
  \mathopen{\copy\@brx\kern-0.5\wd\@brx\usebox{\@brx}}}
\newcommand{\rrangle}[1][]{\savebox{\@brx}{\(\m@th{#1\rangle}\)}%
  \mathclose{\copy\@brx\kern-0.5\wd\@brx\usebox{\@brx}}}
\makeatother

\usepackage{lipsum} 
\usepackage{titlesec}
\titleformat{\subsection}[runin]
       {\normalfont\bfseries}
       {\thesubsection}
       {0.5em}
       {}
       [.]

\newcommand{\A}{\mathfrak{A}}

\newcommand{\B}{\mathfrak{B}} 

\newcommand{\CC}{\mathbb{C}}

\usepackage{tikz-cd}
\usepackage{graphicx}

\def\X{\mathcal{X}}

\def\e{{\sf e}}

\def\BofH{\mathbb B(\mathcal H)}

\def\d{{\rm d}}

\def\({\left(}
\def\[{\left[}
\def\){\right)}
\def\]{\right]}

\def\G{{\sf G}}

\def\H{{\sf H}}

\def\<{\langle}
\def\>{\rangle}
\providecommand{\norm}[1]{\lVert#1\rVert}

 \newtheorem{thm}{Theorem}[section]
 \newtheorem{cor}[thm]{Corollary}
 
 \newtheorem{lem}[thm]{Lemma}
 \newtheorem{prop}[thm]{Proposition}
 \theoremstyle{definition}
 \newtheorem{defn}[thm]{Definition}
 \theoremstyle{remark}
 
 \newtheorem{ex}[thm]{Example}
 \numberwithin{equation}{section}

\numberwithin{equation}{section}

\usepackage[backend=biber, maxnames=10]{biblatex}
\begin{document}


\title{A note on finite-dimensional quotients and the problem of automatic continuity for twisted convolution algebras}

\author{Felipe I. Flores
}

\maketitle

\begin{abstract}
In this note, we will show that the twisted convolution algebra $L^1_{\alpha,\omega}(\G,\A)$ associated with a twisted action of a locally compact group $\G$ on a $C^*$-algebra $\A$ has the following property: Every quotient by a closed two-sided ideal of finite codimension produces a semisimple algebra. Afterward, we use this property, together with results by H. Dales and G.Willis, to extend previous results by the author and to produce large classes of examples of algebras with automatic continuity properties.

\end{abstract}

\bigskip
\bigskip
Much of the progress in the study of automatic continuity in Banach algebras has occurred in connection with the study of group algebras, making it intimately related to the field of abstract harmonic analysis. Examples of this phenomenon can be found in Dales's famous book \cite{Da00} or in the survey \cite{Da78}.

In this note, we will study the continuity of intertwining operators, a goal that has already been carried out in the context of group algebras by Willis \cite{Wi80}, Dales and Willis \cite{DaWi83}, and Runde \cite{Ru96}, among others. Particular versions of this problem have also raised independent interest. For example, we can mention the work of Jewell \cite{Je77} and Willis \cite{Wi92} on automatic continuity for derivations, or the work of Runde \cite{Ru94} on automatic continuity for homomorphisms.

The purpose of this note is to extend previous results on the automatic continuity of the Banach $^*$-algebras given by (generalized, twisted) convolution of $L^1$-functions over groups. More specifically, we seek to improve some of the results in \cite{flores2024} by two fundamentally different but related ways. One involves relaxing the condition of compact generation of the group (fundamental to the results of that paper), while the other one is about making the coefficient algebra finite-dimensional and relaxing the conditions imposed on the group. This allows, of course, great generalizations and new examples of automatic continuity phenomena.

Our approach is based on the study of semisimplicity for the quotients by finite codimensional (also called cofinite), closed, two-sided ideals. This property is surprisingly connected with the theory of automatic continuity, as the results in \cite{Ru96}, and specially in \cite{DaWi83}, exemplify. In fact, our approach will make explicit use of theorems in these papers, attributed to Willis (Theorem \ref{WILLIS}) and Dales-Willis (Theorem \ref{DAWI}). These theorems, combined with results of \cite{flores2024} and the result we obtained on semisimplicity, yield new examples of automatic continuity.

The organization of the paper is as follows: In Section \ref{semis} we introduce what we call twisted convolution algebras and prove that their finite-dimensional quotients are semisimple. This is recorded in Theorem \ref{semisimplicity} and its proof. In Section \ref{mainsec} we recall basic concepts of automatic continuity and then we proceed to combine the above-mentioned results to get our results in automatic continuity, finishing the exposition.

\section{Semisimplicity of finite-dimensional quotients}\label{semis}

A {\it twisted action} is a $4$-tuple $(\G,\alpha,\omega,\A)$, where $\G$ is a locally compact group, $\A$ a $C^*$-algebra and continuous maps $\alpha:\G\to{\rm Aut}({\A})$, $\omega:\G\times\G\to \mathcal{UM}({\A})$, such that $\omega$ and $x\mapsto\alpha_x(a)$ satisfy \begin{itemize}
        \item[(i)] $\alpha_x(\omega(y,z))\omega(x,yz)=\omega(x,y)\omega(xy,z)$,
        \item[(ii)] $\alpha_x\big(\alpha_y(a)\big)\omega(x,y)=\omega(x,y)\alpha_{xy}(a)$,
        \item[(iii)] $\omega(x,\e)=\omega(\e,y)=1, \alpha_\e={\rm id}_{{\A}}$,
    \end{itemize} for all $x,y,z\in\G$ and $a\in\A$. $\e$ denotes the identity in $\G$.

Given such a tuple, one can form the twisted convolution algebra $L^1_{\alpha,\omega}(\G,\A)$, consisting of all Bochner integrable functions $\Phi:\G\to\A$, endowed with the twisted convolution product 
\begin{equation*}\label{convolution}
    \Phi*\Psi(x)=\int_\G \Phi(y)\alpha_y[\Psi(y^{-1}x)]\omega(y,y^{-1}x)\d y
\end{equation*} 
and the involution \begin{equation*}\label{involution}
    \Phi^*(x)=\Delta(x^{-1})\omega(x,x^{-1})^*\alpha_x[\Phi(x^{-1})^*].
\end{equation*} 
With these operations, $L^1_{\alpha,\omega}(\G,\A)$ is a $^*$-Banach algebra under the norm 
$$
\norm{\Phi}_{L^1_{\alpha,\omega}(\G,\A)}=\int_\G\norm{\Phi(x)}_{\A}\,\d x.
$$ 
In these integrals, $\d x$ denotes the Haar measure on $\G$, while $\Delta$ denotes the modular function associated with $\d x$. In the case where $\omega\equiv 1$, we denote the resulting algebra as $L^1_\alpha(\G,\A)$. On the other hand, in the case where $\A=\CC$ and $\alpha\equiv{\rm id}_\CC$, the resulting algebra will be denoted by $L^1_\omega(\G)$ and we will call it a twisted group algebra.

The goal of this chapter is to show that the closed cofinite ideals of $L^1_{\alpha,\omega}(\G,\A)$ produce semisimple quotients, and to do so, we need to introduce a special class of multipliers. It is therefore convenient to recall the definition of the algebra of multipliers of a Banach algebra. 

In what follows, if $\X$ is a Banach space, then $\mathbb B(\X)$ will denote the set of bounded operators $T:\X\to\X$, while ${\rm GL}(\X)\subset\mathbb B(\X) $ will denote the group of bounded operators that are invertible.

\begin{defn}
Let $\B$ be a Banach algebra. A \emph{multiplier} of $\B$ is a pair $m=(\lambda,\mu)$, where $\lambda,\mu\in \mathbb B(\B)$ are such that
$$
a\lambda(b)=\mu(a)b, \quad \lambda(ab)=\lambda(a)b \quad\text{and}\quad \mu(ab)=a\mu(b),
$$
for all $a,b\in\B$.

The set of all multipliers of $\B$ is called the \emph{multiplier algebra} of $\B$ and we denote it by $\mathcal M(\B)$.
\end{defn}

Recall that the product of multipliers is given by the following formula:
\begin{align*}
(\lambda,\mu)(\lambda',\mu')=(\lambda\circ\lambda',\mu'\circ\mu).
\end{align*}
Furthermore, the natural norm in $\mathcal M(\B)$ is given by $\norm{(\lambda,\mu)}_{\mathcal M(\B)}=\max\{\norm{\lambda},\norm{\mu}\}$. If $\B$ is a Banach $^*$-algebra, then the multiplier algebra also has a natural involution, $(\lambda,\mu)^*=(\lambda^*,\mu^*)$, which verifies 
$$
\lambda^*(a)=\mu(a^*)^*\quad\text{ and }\quad\mu^*(a)=\lambda(a^*)^*,\quad\quad\text{for all }a\in\B.
$$
If $\B$ is involutive, then we use $\mathcal{UM}({\B})$ to denote the unitary group of $\mathcal{M}({\B})$.

Note that $\mathcal M(\B)$ is always unital and also contains a copy of $\B$, given by the multipliers $(L_b,R_b)$, $b\in\B$. These multipliers are, naturally, defined by
$$
R_b(a)=ab\quad\text{ and }\quad L_b(a)=ba,\quad\quad\text{for all }a\in\B.
$$

The interesting thing about this inclusion is that, assuming the existence of contractive approximate identities, every non-degenerate representation of $\B$ naturally extends to a representation of $\mathcal M(\B)$. It is a well-known fact that $L^1_{\alpha,\omega}(\G,\A)$ always has a contractive approximate identity, so the following lemma will be of importance to us.

\begin{lem}\label{exten}
Let $\B$ be a Banach algebra, $\X$ be a Banach space, and let $\pi:\B\to \mathbb B(\X)$ be a contractive representation. Assume further that the following are true: \begin{itemize}
\item[(i)] $\B$ has a contractive approximate identity.
\item[(ii)] The representation $\pi$ is non-degenerate, that is, $\overline{\rm span}\{\pi(b)\xi\mid b\in\B, \xi\in \X\}=\X$.
\end{itemize}
Then there exists a unique contractive unital representation $\widetilde\pi: \mathcal M(\B)\to \mathbb B(\B)$, such that $\widetilde\pi\circ\iota_\B=\pi$.
\end{lem}

Given a twisted action $(\G,\alpha,\omega,\A)$, and for $a\in \mathcal{M}(\A),\, y\in\G$, we consider the multiplier $m_{a,y}=(\lambda_{a,y},\mu_{a,y})$ of $L^1_{\alpha,\omega}(\G,\A)$ which is given by 
\begin{align*}
\lambda_{a,y}(\Phi)(x)&=a\alpha_y\big(\Phi(y^{-1}x)\big)\omega(y,y^{-1}x), \\
\mu_{a,y}(\Phi)(x)&=\Delta(y^{-1})\Phi(xy^{-1})\alpha_{xy^{-1}}(a)\omega(xy^{-1},y).
\end{align*}
We also establish the following notation:

$$
\Gamma_{\G,\A}= \{m_{u,y}\mid u\in \mathcal{UM}(\A),\, y\in\G\}.
$$
In the following lemma, we will compile some well-known and easy-to-prove facts, but which are useful for the proofs of our results.

\begin{lem}\label{easy} The following statements are true.
\begin{enumerate}
\item[(i)] $\Gamma_{\G,\A}$ is a group.
\item[(ii)] Every $m_{u,y}\in \Gamma_{\G,\A}$ is unitary and has norm $1$.
\item[(iii)] Every multiplier of the form $m_{a,y}$ can be written as a linear combination of $4$ elements in $\Gamma_{\G,\A}$.
\item[(iv)] The adjoint of $m_{a,y}$ satisfies the formula \begin{equation}\label{inv}
m_{a,y}^{*}=m_{\omega(y^{-1},y)^*\alpha_{y^{-1}}(a^*),y^{-1}},
\end{equation} for all $a\in\mathcal{M}(\A), y\in\G$.
\end{enumerate}
\end{lem}

Next, we proceed to prove the main result of this section. Our proof is based on the fact that representations of compact groups are similar to unitary representations. The relevant fact is the following (see \cite[Theorem 0.1]{Pi05}).

\begin{lem}\label{dix}
Let $\X$ be a finite-dimensional Hilbert space and $V\subset {\rm GL}(\X)$ be a subgroup such that $\sup_{v\in V}\norm{v}_{\BofH}<\infty$. Then there exists a positive and invertible linear transformation $T\in {\rm GL}(\X)$ such that $TvT^{-1}\in \mathcal U(\X)$, for all $v\in V$.
\end{lem}

\begin{thm}\label{semisimplicity}
Let $(\G,\alpha,\omega,\A)$ be a twisted action. If $I\subset \B=L^1_{\alpha,\omega}(\G,\A)$ is a closed, finite-codimensional two-sided ideal, then $I$ is automatically self-adjoint and the quotient algebra $\B/I$ is semisimple.
\end{thm}
\begin{proof}
Since $I$ is closed and finite-codimensional, $\mathcal X=\B/I$ is a finite-dimensional Banach space. Let $\langle\cdot,\cdot\rangle$ be any inner product; since it is finite-dimensional, $\X$ is a Hilbert space with respect to this inner product.

We denote by $\pi:\B\to\mathbb B(\mathcal X)$ the induced representation on the quotient, that is, 
$$
\pi(\Phi)(\Psi+I)=\Phi*\Psi+I,
$$
for all $\Phi,\Psi\in \B$. This representation is contractive and non-degenerate, so, due to Lemma \ref{exten} and abusing the notation, $\pi$ extends to $\mathcal M(\B)$ and, hence, the operators $\pi(m_{a,y})\in\mathbb B(\X)$ are well-defined and uniformly bounded. In fact, it is not difficult to notice that they satisfy the identity
$$
\pi(m_{a,y})(\Psi+I)=m_{a,y}(\Psi)+I,\quad\quad\text{ for all }\Psi\in\B.
$$
For this reason, one observes that
\begin{equation}\label{desinte}
\pi(\Phi)(\Psi+I)=\int_\G \pi(m_{\Phi(y),y})(\Psi+I)\d y=\int_\G m_{\Phi(y),y}(\Psi)\d y +I.
\end{equation}
Now, we note that $V=\{\pi(m)\}_{m\in\Gamma_{\G,\A}}$ satisfies all the conditions of Lemma \ref{dix} and, therefore, there must exist a positive and invertible operator $T\in {\rm GL}(\X)$ such that $T\pi(m)T^{-1}\in \mathcal U(\X)$, for all ${m\in\Gamma_{\G,\A}}$.

We then define the representation $\pi':\B\to\mathbb B(\X)$ given by $\pi'(\Phi)=T\pi(\Phi)T^{-1}$ and we will now prove that it is a $^*$-representation and that ${\rm Ker}\,\pi'=I$, from which it will follow that $\B/I$ is $^*$-isomorphic to $\pi'(\B)$, which is a $C^*$-algebra, and therefore we will have shown that $\B/I$ is semisimple.

Indeed, note that ${\rm Ker}\,\pi'={\rm Ker}\,\pi$. Now, let $\Phi\in{\rm Ker}\,\pi$ and let $\Psi_j\in\B$ be some approximate bounded identity of $\B$. We note that
$$
I=\lim_j\pi(\Phi)(\Psi_j+I)=\lim_j \Phi*\Psi_j+I=\Phi+I,
$$ 
so $\Phi\in I$. This proves that ${\rm Ker}\,\pi'=I$.

Now let's see that $\pi'$ is a $^*$-representation. Indeed, if $m\in\Gamma_{\G,\A}$ and $\xi,\eta\in\X$, then one has 
\begin{align*}
    \langle \pi'(m)\xi,\eta\rangle&=\langle \xi,\pi'(m)^{*}\eta\rangle \\
    &=\langle \xi,\pi'(m)^{-1}\eta\rangle \\
    &=\langle \xi,\pi'(m^{-1})\eta\rangle=\langle \xi,\pi'(m^*)\eta\rangle.
\end{align*}

But, remembering that every $m_{a,y}$ can be written as a linear combination of $4$ elements in $\Gamma_{\G,\A}$ (point \emph{(iii)} of Lemma \ref{easy}), we see that
$$
\pi'(m_{a,y})^*=\pi'(m^*_{a,y}),\quad\text{ for every }a\in\mathcal{M}(\A),y\in \G.
$$
And, consequently, for $\Phi\in \B,\xi\in\X$, and using the equality \eqref{desinte}, one observes that
\begin{align*}
    \pi'(\Phi^*)\xi=T\pi(\Phi^*)T^{-1}\xi&=T\int_\G \pi\big( m_{\Phi^*(y),y}\big)T^{-1}\xi\,\d y \\
    &= T\int_\G \Delta(y^{-1})\pi\big(m_{\omega(y,y^{-1})^*\alpha_y(\Phi(y^{-1})^*),y}\big)T^{-1}\xi\,\d y \\
    &=T\int_\G \pi\big(m_{\omega(y^{-1},y)^*\alpha_{y^{-1}}(\Phi(y)^*),y^{-1}}\big)T^{-1}\xi\,\d y \\
    &\overset{\eqref{inv}}{=}\int_\G T\pi\big(m_{\Phi(y),y}^*\big)T^{-1}\xi\,\d y \\
    &=\int_\G \pi'(m_{\Phi(y),y})^*\xi\,\d y=\pi'(\Phi)^*\xi,
\end{align*}
 which finishes the proof.
\end{proof}

\section{Applications to the problem of automatic continuity}\label{mainsec}

Let $\B$ be a Banach algebra. A Banach space $\mathcal X$ which is also a $\B$-bimodule is called a \emph{Banach $\B$-bimodule} if the maps $$\B\times \mathcal X\ni(b,\xi)\mapsto b\xi \in\mathcal X \quad\text{ and }\quad \X\times\B\ni(\xi,b)\mapsto \xi b \in\mathcal X$$ are jointly continuous. 

\begin{defn}
    Let $\B$ be a Banach algebra and $\X_1,\X_2$ be a Banach $\B$-bimodule. A linear map $\theta:\X_1\to\X_2$ is called a \emph{$\B$-intertwining operator} if for each $b\in\B$, the maps $$\X_1\ni \xi\mapsto \theta(b\xi)-b\theta(\xi)\in \X_2\quad\text{ and }\quad \X_1\ni \xi\mapsto \theta(\xi b)-\theta(\xi)b\in \X_2$$ are continuous.
\end{defn}

\begin{ex}\label{ex-inter}\begin{enumerate}
        \item[(i)] Every $\B$-bimodule homomorphism between Banach $\B$-bimodules is a $\B$-intertwining operator.
        \item[(ii)] Let $\X$ a Banach $\B$-bimodule. A derivation is a linear map $D:\B\to \X$ satisfying $$D(ab)=D(a)b+aD(b).$$ Every derivation is a $\B$-intertwining operator. 
    \end{enumerate}
\end{ex}

The automatic continuity problem consists in understanding what kind of conditions guarantee that every $\B$-intertwining operator on the Banach algebra $\B$ is necessarily continuous. A fundamental tool for attacking this problem is the so-called continuity ideal, which we introduce below.

\begin{defn}
    Let $\B$ be a Banach algebra, and $\theta:\X_1\to\X_2$ a $\B$-intertwining operator between Banach $\B$-bimodules. Then $$\mathscr I(\theta)=\{b\in\B\mid \X_1\ni \xi\mapsto \theta(b\xi)\in \X_2 \text{ and }\X_1\ni \xi\mapsto \theta(\xi b)\in \X_2\text{ are continuous}\}$$ is the \emph{continuity ideal} of $\theta$.
\end{defn}

Note that $\mathscr I(\theta)$ is closed, since $\X_2$ is a Banach $\B$-bimodule. The following theorem is due to Willis \cite[Lemma 4.3.5]{Wi80}. See also \cite[pag. 498]{Ru96}.
\begin{thm}[Willis]\label{WILLIS}
    Let $\B$ be a Banach algebra, $\X_1,\X_2$ Banach $\B$-modules and $\theta:\X_1\to\X_2$ an $\B$-intertwining operator. Suppose that there is a directed family $\{\B_i\}_i$ of Banach subalgebras of $\B$ such that \begin{enumerate}
        \item $\B=\overline{\bigcup_i \B_i}$ and
        \item for each index $i$, the algebra $\B_i/\B_i\cap \mathscr I(\theta)$ is semisimple and finite-dimensional.
    \end{enumerate} Then $\mathscr I(\theta)$ has finite codimension in $\B$.
\end{thm}

The main application of this result is to lift the compact generation hypothesis of $\G$ from some of the results obtained in \cite{flores2024}. We would like to emphasize that this restriction was of fundamental importance in that work, since it allowed us to guarantee the existence of weight functions with remarkable properties (see \cite[Lemma 3.4]{flores2024}).

\begin{prop}\label{lifing}
    Let $(\G,\alpha,\omega,\A)$ be a twisted action, denote $\B=L^1_{\alpha,\omega}(\G,\A)$ and suppose we are given an $\B$-intertwining operator $\theta:\X_1\to\X_2$ between Banach $\B$-bimodules $\X_1,\X_2$ with the property that for all open, compactly generated subgroups $\H\subset \G$, the ideal $\mathscr I(\theta)\cap L^1_{\alpha,\omega}(\H,\A)$ has finite codimension in $L^1_{\alpha,\omega}(\H,\A)$, then $\mathscr I(\theta)$ has finite codimension in $\B$.
\end{prop}
\begin{proof}
    We note that $\mathscr I(\theta)\cap L^1_{\alpha,\omega}(\H,\A)$ coincides with the ideal of continuity of $\theta$, when $\theta$ is considered as an $L^1_{\alpha,\omega}(\H,\A)$-intertwining operator and it is therefore closed. Now, we consider the family $\{\H_i\}_i$ of compactly generated open subgroups of $\G$, ordered by inclusion, and note that the family $\B_i=L^1_{\alpha,\omega}(\H_i,\A)$ is a directed family of subalgebras of $\B$ such that $\B=\overline{\bigcup_i \B_i}$. The latter follows, for example, from the fact that $\bigcup_i \B_i$ contains all continuous functions of compact support.

Note that $\B_i/\B_i\cap \mathscr I(\theta)$ is finite-dimensional by assumption and semisimple by the Theorem \ref{semisimplicity}. Then, the result follows from applying the Theorem \ref{WILLIS}.
\end{proof}

In particular, we are now able to provide the following examples of automatic continuity.

\begin{cor}\label{ext}
    Let $\G$ be a nilpotent locally compact group. Let $\X$ be a Banach $\B$-bimodule and $\theta:\B\to\X$ a $\B$-intertwining operator. Then $\theta$ is automatically continuous for the following choices of $\B$: \begin{enumerate}
        \item[(i)] Twisted group algebras $L^1_\omega(\G)$, associated with a $2$-cocycle $\omega:\G\times \G\to \CC$.
        \item[(ii)] Convolution algebras $\ell^1_\alpha(\G,\A)$, where $(\G,\alpha,\A)$ is a $C^*$-dynamical system with $\A$ being a unital nuclear $C^*$-algebra.
        \end{enumerate}
\end{cor}
\begin{proof}
    By combining Proposition \ref{lifing} with \cite[Corollary 4.21]{flores2024}, we know that $\mathscr I(\theta)\subset \B$ is a closed cofinite ideal. Moreover, in both cases the algebra $\B$ has the following property: every closed and cofinite two-sided ideal $I\subset\B$ has a bounded norm left approximate identity.

This property just mentioned is proved directly in the first case \cite[Theorem A.3]{flores2024} and follows from the combination of \cite[Proposition VII.2.31]{He89} with the fact that $\ell^1_\alpha(\G,\A)$ is amenable \cite[Proposition IV.4.2]{Je99} in the second.

That said, we can repeat part of the argument in \cite[Theorem 3.6]{flores2024} to conclude the proof. Indeed, due to the Cohen-Hewitt factorization theorem \cite[Corollary 11.12]{BD73}, for every sequence $\{b_n\}\subset\mathscr I(\theta)$ that converges to zero, there exist $c, d_n\in\mathscr I(\theta)$ that factor $b_n$: 
$$
b_n=cd_n\quad\text{y}\quad \lim_{n} d_n=0.
$$ 
Since the function $\B\ni d\mapsto \theta(cd)$ is continuous by the definition of $\mathscr I(\theta)$, we have 
$$
\lim_n\theta(b_n)=\lim_n\theta(cd_n)=0
$$ 
and therefore the restriction of $\theta$ to $\mathscr I(\theta)$ is continuous. Since $\mathscr I(\theta)$ has finite codimension, $\theta$ is in fact continuous throughout $\B$.
\end{proof}

Now we will restrict ourselves to the study of (some) finite-dimensional valued intertwining operators; this will provide more flexibility on the assumptions about the twisted action. Dales and Willis proved the following theorem in \cite[Theorem 2.5]{DaWi83} and it will be our main motivation for what follows.

\begin{thm}[Dales-Willis]\label{DAWI}
    Let $\B$ be a Banach algebra such that $\B/I$ is semisimple for every cofinite, closed two-sided ideal $I\subset \B$. Then the following conditions are equivalent\begin{enumerate}
        \item[(i)] Each homomorphism from $\B$ with finite-dimensional range is continuous.
        \item[(ii)] Every derivation into a finite-dimensional Banach $\B$-bimodule is continuous.
        \item[(iii)] Every cofinite two-sided ideal of $\B$ is closed.
        \item[(iv)] $I^2$ is closed and cofinite, for every closed, cofinite two-sided ideal $I\subset \B$.
    \end{enumerate}
\end{thm}

And so an application of Theorem \ref{semisimplicity} yields the following proposition.
\begin{prop}
    Let $(\G,\alpha,\omega,\A)$ a twisted action. Then all conditions in Theorem \ref{DAWI} are equivalent for $L^1_{\alpha,\omega}(\G,\A)$.
\end{prop}

In particular, we obtain many classes of examples for this finite-dimensional phenomena. We compile them in the following corollary. As we shall see next, this restriction we just imposed on ourselves to finite-dimensional valued intertwining operators allows us to greatly extend previous results by allowing us to drop the assumptions on the group (cf. \cite[Corollary 4.21]{flores2024}).
\begin{cor}
    Let $\B$ be one of the following algebras. \begin{enumerate}
        \item[(i)] $L^1_\omega(\G)$, for an amenable group $\G$ and a $2$-cocycle $\omega:\G\times \G\to \CC$.
        \item[(ii)] $\ell^1_\alpha(\G,\A)$, for an (untwisted) action $(\G,\alpha, \A)$ where $\G$ is discrete and amenable and $\A$ is a nuclear $C^*$-algebra.
    \end{enumerate} Then $\A$ satisfies all the conditions in Theorem \ref{DAWI}.
\end{cor}
\begin{proof}
   We verify condition \emph{(iv)} of theorem \ref{DAWI}. As in the proof of Corollary \ref{ext}, we see that every closed cofinite two-sided ideal $I\subset\B$ has a bounded left approximate identity. In this case, $I=I^2$ also follows from the Cohen-Hewitt factorization theorem.
\end{proof}

\section*{Acknowledgments}

This work was supported by NSF project DMS-2144739. The author sincerely thanks Professor Ben Hayes for the many interesting discussions surrounding this topic. The author is also grateful to Diego Jauré, Moria Labraña, and the reviewers for their helpful comments on earlier versions of the article.

\printbibliography

\bigskip
\bigskip
ADDRESS

\smallskip
Felipe I. Flores

Department of Mathematics, University of Virginia,

114 Kerchof Hall. 141 Cabell Dr,

Charlottesville, Virginia, United States

E-mail: hmy3tf@virginia.edu

\end{document}